\documentclass[12pt]{article}
\usepackage{amsmath, amsfonts, amsthm, amssymb, color, hyperref, extarrows, enumitem}

\newtheorem{theorem}{Theorem}[section]

\newtheorem{definition}[theorem]{Definition}
\newtheorem{cor}[theorem]{Corollary}

\newtheorem{pro}[theorem]{Proposition}
\DeclareMathOperator{\Real}{Re}
\DeclareMathOperator{\Imagine}{Im}
\numberwithin{equation}{section}

\usepackage[T1]{fontenc}

\newtheorem{klem}[theorem]{Key Lemma}

{\theoremstyle{definition}
\newtheorem{example}[theorem]{Example}}

\usepackage[titletoc]{appendix}

\usepackage{cite}
 
\hypersetup{
	colorlinks   = true,
	citecolor    = magenta}

	\usepackage[titletoc]{appendix}

\begin{document}
\title{\vspace{-1.2cm}  \bf  Equality in Suita's conjecture and metrics of constant Gaussian curvature \rm}
\author{Robert Xin Dong}
\date{}
\maketitle

\begin{abstract} 
Without using the $L^2$ extension theorem, we provide a new proof of the equality part in Suita's conjecture, which states that for any open Riemann surface admitting a Green's function, the Bergman kernel and the logarithmic capacity coincide at one point if and only if the surface is biholomorphic to a disc possibly less a relatively closed polar set. In comparison with Guan and Zhou's proof, our proof essentially depends on Maitani and Yamaguchi's variation formula for the Bergman kernel, and we explore the harmonicity in such variations. As applications, we characterize the above surface by the constant Gaussian curvature property of the Bergman kernel or metric, and also find the relations with disc quotients. Additionally, we obtain results on planar domains without the Bergman-completeness assumption.

\end{abstract}

\renewcommand{\thefootnote}{\fnsymbol{footnote}}
\footnotetext{\hspace*{-7mm} 
\begin{tabular}{@{}r@{}p{16.5cm}@{}}
& Keywords. Bergman kernel, Bergman metric, Bergman representative coordinate, Green's function, K\"ahler-Einstein metric, logarithmic capacity, open Riemann surface, Suita conjecture, Suita metric\\
& Mathematics Subject Classification. Primary 32A25; Secondary 30C40, 31A15, 30C85, 32D15
\end{tabular}}

\section{Introduction}

Let $X$ be a Riemann surface, and let $w$ be a local coordinate on a neighborhood $V$ of $z_0 \in X$ satisfying $w(z_0) = 0$.
 Let $\kappa_X$ be the Bergman kernel for holomorphic (1,0) forms on $X$.
We define
$$ 
K(z)|dw|^2 := \kappa_X(z)|_{V }.
$$ 
If $X$ is an open Riemann surface admitting a Green's function $G$, 
 the logarithmic capacity $c_{\beta}$   is locally defined as
\begin{equation} \label{cap}
c_{\beta}(z_0):=\exp \lim_{z\to {z_0}} \left\{G(z, z_0)- \log |w(z)|\right\}.
\end{equation}
Suita in \cite{Su} asked about the precise relations between the above two quantities and  conjectured that {\it on any open surface $X$ as above,   
\begin{equation} \label{leq}
\pi K(z_0) \geq c_{\beta}^2 (z_0),
\end{equation} 
and equality in \eqref{leq} holds true at $z_0\in X$ if and only if $X$ is biholomorphic to a disc possibly less a closed set of inner capacity zero}.  A closed set of inner capacity zero is a polar set (local singularity set of a subharmonic function). 

\medskip

The relations between Suita's conjecture and the Ohsawa-Takegoshi $L^2$ extension theorem \cite{OT} were first observed by Ohsawa \cite{O95, O98} and contributed by a number of mathematicians (see the works of Siu \cite{Si}, Berndtsson \cite{Be0}, Chen \cite{Ch}, B\l{}ocki \cite{Bl07} and Guan-Zhou-Zhu \cite{GZZ, GZZ1}). The $L^2$ extension theorem has many important applications in complex algebraic and differential geometry and several complex variables. For example, it was explored to prove the famous invariance of plurigenera \cite{Si2, Si02, Pa}, Demailly's approximation \cite{Dem92, DK, Dem12}, strong openness conjecture \cite{GZ4, GZ5}, etc. See also the recent survey papers of B\l{}ocki \cite{Bl}, Ohsawa \cite{O20, O20B}, Zhou \cite{Z}, and the references therein.

\medskip

For more than forty years, the Suita conjecture remained unsolved until the breakthrough by B\l{}ocki \cite{Bl13}, who proved the inequality part of the conjecture for bounded planar domains. Later, Guan and Zhou \cite{GZ1, GZ3} proved both the inequality and equality parts of the conjecture for open Riemann surfaces. The proofs in \cite{Bl13, GZ1, GZ3} were based on the optimal constant versions of the $L^2$ extension theorems, obtained by introducing the crucial auxiliary weight functions. Afterwards, in \cite{Bl14}, B\l{}ocki gave a much simpler proof of the inequality part of the Suita conjecture using the tensor power trick. In connection with this, Lempert discovered that B\l{}ocki's main result in \cite{Bl14} could also follow from a variational approach as sketched in \cite{Bl15}. Eventually, such a completely different and remarkable approach was further developed in \cite{BL} by Berndtsson and Lempert, who gave yet another proof of the inequality part of the Suita conjecture using the plurisubharmonic variation properties of Bergman kernels \cite{MY, B1}, and more strikingly deduced the sharp Ohsawa-Takegoshi extension theorem from Berndtsson's important positivity result on the direct image bundles \cite{B2}. Conversely, Guan and Zhou in \cite{GZ3} showed that the sharp $L^2$ extension theorem gives a proof of the above variation properties of Bergman kernels. Last but not least, in \cite{Bl17}, B\l{}ocki gave a purely one-dimensional self-contained proof of the inequality part of the conjecture.

\medskip

The equality part of the Suita conjecture is much harder. To the author's knowledge, there is no published proof except Guan and Zhou's one in \cite{GZ3}, where a very general form of the $L^2$ extension theorem with sharp constant was used. We believe that it is fundamentally interesting to understand this beautiful one dimensional complex analysis result. Recently, it was observed in \cite{DTZ} that in Guan and Zhou's final step of proving the conjecture, the use of the analytic capacity is not necessary.

 \medskip

Our first motivation  is to provide a new proof of the Equality Part in Suita's conjecture, without using the powerful $L^2$ extension theorem. Inspired by Berndtsson and Lempert's work in \cite{BL}, our proof essentially depends only on Maitani and Yamaguchi's variation formula for the Bergman kernel in \cite{MY}. In fact, the variational idea was already suggested in Suita's original paper \cite[Sec. 6]{Su}. 
For an open Riemann surface $X$ admitting a Green's function $G(z, z_0)$ with a pole $z_0\in X$, 
and for each fixed $\tau\in \mathbb C$ such that $\Real \tau \leq 0$, define the sub-surface 
\begin{equation} \label{X_tau} 
X_{\tau}:=\{z\in X\, | \, G(z, z_0)<\Real \tau\},
\end{equation} 
and let $\kappa_{X_\tau}$ be the Bergman kernel for holomorphic (1,0) forms on $X_\tau$.
Let $w$ be a local coordinate on a neighborhood $V $ of $z_0 $ satisfying $w(z_0) = 0$, and we define  
$$ 
K_\tau(z)|dw|^2 := \kappa_{X_\tau}(z)|_{V }.
$$  

\begin{theorem} \label{main} Let $X$ be an open Riemann surface admitting a Green's function.   
Then, \eqref{leq} holds true. Moreover, the following conditions are equivalent:
\begin{enumerate}

\item [$(i)$] $\pi K(z_0) = c_{\beta}^2(z_0)$;

\item [$(ii)$] $X$ is biholomorphic to a disc possibly less a relatively closed polar set;

\item [$(iii)$] the minimal $L^2$ norm of holomorphic 1-forms $F$ on $X$ such that $F|_{z_0}=  c_{\beta}(z_0) dw|_{z_0} $ and $||F||_{L^2}\leq \sqrt {\pi}$ is $\sqrt {\pi}$; 

\item [$(iv)$] $\log K_{\tau}(z_0)$ is harmonic in $\tau$, and equivalently, linear in $\Real \tau$.

\end{enumerate}
\end{theorem}

  \medskip

\noindent{}{\bf Remarks.}

\begin{enumerate}[label=(\roman*)]
 
\item [(a)] 
The equivalence between (i) and (ii) is exactly Guan and Zhou's theorem in \cite{GZ3}.
It can be easily seen that Condition (ii) implies the other three conditions above, since the possible polar part is negligible for $L^2$ holomorphic 1-forms.

\item [(b)]  In this paper, we first show that (i) is equivalent to each of (iii) and (iv). Then, we will give a purely variational proof of the fact that  (i) implies (ii) as follows:
 
\begin{equation} \label{circle}
 (i) \Longrightarrow
  \begin{array}{ll}
(iii)\\
\,\,+\\
(iv)\\
\end{array} \Longrightarrow \text{Key Lemma \ref{off}} \Longrightarrow (ii). 
\end{equation}

\item [(c)]  The applications of Theorem \ref{main} include the resolution of the so-called $L$-conjecture and the extended Suita conjecture, which was also given in \cite{GZ3}. For any bounded domain $D$ in $\mathbb C^n$, it is well known that at any $z\in D$, $K(z) \geq v^{-1}(D)$, where $v(\cdot)$ is the Euclidean volume. As another application of Theorem \ref{main}, the author and Treuer in \cite{DT} found a new characterisation of planar domains by the minimal point of the Bergman kernel, and proved that
  $
K(z_0)=v^{-1}(D)
$ at some $z_0\in D$
if and only if 
$
D$ is a disk centered at $z_0$ possibly less a relatively closed polar set. See also \cite{DTZ}.

\end{enumerate}

Our second motivation  is to give more applications of the Suita type result in the viewpoint of curvature properties. In the same paper \cite{Su}, Suita proved the identity
\begin{equation} \label{Curv} 
\frac{\partial ^2 (\log c_{\beta} )}{\partial w \partial \bar w}=\pi K,
\end{equation}
which gives the curvature interpretation of his conjecture for any hyperbolic Riemann surface.  
The Inequality Part in \eqref{leq} is equivalent to the fact that the Gaussian curvature of the Suita metric $c_{\beta}^2(z)|dw|^2$ always has an upper bound $-4$;
the Equality Part, which says that the curvature attains $-4$ at some $z_0 \in X$, guarantees that the surface $X$ necessarily satisfies Condition (ii). In the latter case, if the metric is additionally assumed to be complete, then $X$ must be simply connected. 

  \medskip

The Gaussian curvature of the Bergman kernel $K(z)|dw|^2$ on any Riemann surface satisfying Condition (ii) is    
identically $-4\pi$. Our next theorem shows that the constant curvature property determines the surface  up to biholomorphism.

\begin{theorem} \label{2nd} Let $X$ be a Riemann surface. Then, the Gaussian curvature of the Bergman kernel $K(z)|dw|^2$ 
is identically $-4\pi$  
if and only if $X$ is biholomorphic to a disc possibly less a relatively closed polar set.

  \end{theorem} 

The equivalent condition in Theorem \ref{2nd} further implies that the Gaussian curvature of the Bergman metric 
$g (z)|dw|^2$, defined in \eqref{metric}, is identically $-2$.
In \cite{Lu}, Lu proved his well-known uniformization theorem: a bounded domain in $\mathbb C^n$ with a complete Bergman metric of constant holomorphic sectional curvature is biholomorphic to the Euclidean ball. A general problem raised by Yau \cite[pp. 679]{Yau} is to characterize manifolds whose Bergman metrics are K\"{a}hler-Einstein. 
When the manifold is a smoothly bounded strictly pseudoconvex domain, Cheng conjectured that the Bergman metric is K\"{a}hler-Einstein if and only if the domain is biholomorphic to the ball. 
After the previous works of Fu and Wong \cite{FW} and Nemirovski and Shafikov \cite{NS}, the Cheng conjecture was then confirmed in \cite{HX} by Huang and Xiao, who later extended the Lu type uniformization theorem to Stein spaces, of dimension $\geq 2$, with possibly isolated normal singularities (see \cite{HX2}, and further works \cite{EXX, HL} on the Bergman-Einstein metrics on those Stein spaces). In the proofs of the above mentioned important works, the completeness of the Bergman metric played a decisive role.  On the other hand, the author and Wong recently in \cite{DW} extended Lu's theorem towards the Bergman-incomplete situation and characterized pseudoconvex domains that are biholomorphic to a ball possibly less a relatively closed pluripolar set.  Using the Suita type result and the idea of \cite{DW}, we completely drop the completeness assumption and characterize  one dimensional Stein manifolds by the curvature property of their Bergman kernel or metric.  Our next theorem is stated as follows.

\begin{theorem} \label{3rd} Let $X$ be a Riemann surface whose Bergman metric $g (z)|dw|^2$ has  Gaussian curvature identically equal to $-2$. Then $X$ is biholomorphic to a disc possibly less a relatively closed polar set, if either of the following conditions holds true:
\begin{enumerate}

\item [$1$.] there exists $p \in X$ such that $g(p) \geq 2 c^2_\beta (p)$;

\item [$2$.] there exists $p \in X$ at which the Gaussian curvature of the Bergman kernel $K(z)|dw|^2$ is less than or equal to $-4\pi$, i.e., $g(p) \geq 2 \pi K (p)$.

\end{enumerate}

  \end{theorem}

 Theorem \ref{2nd} follows from Theorem \ref{3rd}. In both theorems,  the surface is not required, a priori, to be non-compact, but later is shown to be necessarily hyperbolic;  moreover, the above constant Gaussian curvature conditions cannot be replaced by the curvature equaling a certain constant at one point as in Theorem \ref{main}, cf. \cite{Bl13, Bl15}. Due to Yau's Schwarz lemma \cite{Yau78} and the inequality part \eqref{leq} of the Suita conjecture, 
  Condition 1 in Theorem \ref{3rd} holds true whenever the Gaussian curvature of the complete Bergman metric has a  lower bound $-2$. Namely, we get 
 
 \begin{theorem} \label{lower} Let $X$ be a hyperbolic Riemann surface whose Bergman metric $g (z)|dw|^2$ is complete and has Gaussian curvature no less than $-2$. Then, $g  \geq 2 c^2_\beta $ on $X$.
 
   \end{theorem}

 Theorems \ref{3rd} and \ref{lower} immediately yield

 \begin{cor} \label{cor} For a  Riemann surface $X$, the Bergman metric is complete and has Gaussian curvature identically equal to  $-2$ if and only if $X$ is biholomorphic to a disc.

  \end{cor}

  For planar domains,  we obtain  additional results without any Bergman-completeness assumption.

 \begin{theorem} \label{4th} Let $D$ be a domain in $\mathbb C$ and let $I \subset  \partial D$ denote the set of possible irregular boundary points  with respect to the Dirichlet problem.
  If $ D \cup I$ is an open set and the Gaussian curvature of the Bergman metric on $D$ is identically $-2$, then $D$ is biholomorphic to a disc possibly less a relatively closed polar set.

  \end{theorem} 
  
 We remark that for general planar domains, the above $ D \cup I$ is not always an open set (see Example \ref{Zal}). Another  characterization of planar  domains that are biholomorphic to a disc possibly less a relatively closed polar set is stated as follows.

 \begin{theorem} \label{boundary point} 
     Let $D \subset \mathbb C$ be a domain  whose Bergman kernel is denoted by $K(z)$. Assume no boundary point $q\in \partial D$ satisfies 
     \begin{equation} \label{=>} 
     \limsup_{z\to q} K(z) = \infty  >   \liminf_{z\to q} K(z).
\end{equation} 
     Then, the Gaussian curvature of the Bergman metric on $D$ is identically $-2$ if and only if
 Conditions (i)--(iv) in Theorem \ref{main} hold  true for $D$.
     
  \end{theorem}

Theorems \ref{4th} and \ref{boundary point}  improve Theorem 1.3 in \cite{DW} for the case $n=1$ by replacing the Condition ($\star$) there.  Theorem \ref{3rd}, Corollary \ref{cor} as well as Theorems \ref{4th} and \ref{boundary point} can be viewed as one dimensional cases of Yau's problem of which the Bergman metric having constant Gaussian curvature is equivalent to being K\"{a}hler-Einstein.  
 For finite subgroups $\Gamma \subset \text{Aut} (\mathbb D)$, Huang and Li in \cite{ HL} made explicit computations on disc quotients $\mathbb D / \Gamma$, 
and showed that the Bergman metrics on their regular parts are K\"{a}hler-Einstein. We further observe that the same thing happens to their Bergman kernels, and consequently conclude by Theorem \ref{2nd} that the regular parts of $\mathbb D / \Gamma$ are necessarily biholomorphic to bounded planar domains.

\begin{cor} \label{quotient} For any nontrivial finite subgroup $\Gamma \subset \text{\normalfont Aut} (\mathbb D)$, the regular part of $ \mathbb D / \Gamma $ is biholomorphic to $\mathbb D \setminus \{0\}.$

  \end{cor}

The organization of this paper is as follows. 
 In Sec. \ref{Sec2} we use the variational approach to show the equivalence among Conditions (i), (iii) and (iv) in Theorem \ref{main}. In Sec. \ref{Sec3} we prove Key Lemma \ref{off} and use it to complete the proof of Theorem \ref{main}. In Sec. \ref{Sec4} we demonstrate our results on the Bergman metric of constant Gaussian curvature for Riemann surfaces. In Sec. \ref{Sec5} we complete the proofs of Theorems \ref{2nd}, \ref{3rd}, \ref{lower}, \ref{4th} and \ref{boundary point}.  Lastly,   in Sec. \ref{Sec6} we apply our results to disc quotients.

\section{Variational proofs of (i) $\Longleftrightarrow$ (iii), and (i) $\Longleftrightarrow$ (iv)} \label{Sec2}

The definition of the Bergman kernel is recalled as below.
Take holomorphic 1-forms $\alpha, \beta$ on a Riemann surface $X$, and define their inner product
$
\langle \alpha, \beta \rangle:= \frac{\sqrt{-1}}{2} \int_X \alpha \wedge \bar  \beta.
$
We put $\left \| \alpha \right \|_{L^2}:=\langle \alpha, \alpha\rangle^{1/2}.$ Consider the Hilbert space of holomorphic 1-forms $\alpha$ such that $\left \| \alpha \right \|_{L^2}<\infty$, and take its complete orthonormal basis $\{e_1, e_2, ... , e_j, ... \}$. Then, at $z, s \in X$, the Bergman kernel on $X$ is defined as
$$
\kappa_X (z, \overline{s}):=\sum_j e_j(z)\wedge \overline{e_j(s)}.
$$ 
The Bergman kernel is independent of the choice of the basis and is a reproducing kernel of the space of $L^2$ holomorphic top-forms. For simplicity, we write $\kappa_X (z):=\kappa_X (z, \overline{z})$.

Let $\pi: \mathcal X \to \mathbb C \ni \tau$ be a holomorphic family of Riemann surfaces such that the total space $\mathcal X$ is a complex $2$-dimensional manifold.
Let $\kappa_{X_\tau}$ be the fiberwise Bergman kernel for holomorphic (1,0) forms on the non-exceptional Riemann surface $X_\tau$, and define $$ 
K_\tau(z)|dw|^2 := \kappa_{X_\tau}(z)|_{V }, \quad
 K_\tau(z, \overline{s}) dw\otimes d\overline{t} := \kappa_{X_\tau}(z, \overline{s})|_{V\times \overline{U} },
$$ 
where $(V\times \overline{U},(w, \overline{t} ))$ 
is some local coordinate system for $X_\tau\times\overline{X_\tau}$.
The following variational result is due to Maitani and Yamaguchi \cite{MY} (see also \cite{B1, B2}).

\begin{theorem} \label{M-Y} If $\mathcal X$ is Stein, then $\log K_{\tau}(z)$ is plurisubharmonic in $(\tau, z)$. \end{theorem}

In particular, $\log K_{\tau}(z)$ is subharmonic in $\tau$, for any fixed $z$.

  \medskip

For an open Riemann surface $X$ admitting a Green's function $G(z, z_0)$ with a pole at $z_0\in X$, we apply Theorem \ref{M-Y} to the complex $2$-dimensional Stein manifold defined as
\begin{equation} \label{total} \mathcal X:=\{(z, \tau)\in  X\times \mathbb C \,|\,G(z, z_0)<\Real \tau\}.
\end{equation} 
In this case, each fiber $X_{\tau}$ is a Riemann surface defined in \eqref{X_tau}. If $\tau=0$, then $X_{0}$ is just the original surface $X$, whose Bergman kernel is $\kappa_{X}= \kappa_{X_0}$.

  \medskip

The Suita metric $c_{\beta}^2(z)|dw|^2 $ is a 2-form  associated to the  logarithmic capacity $c_{\beta}$ which is locally defined by \eqref{cap}.
Let $\mathcal H$ denote the Hilbert space of square integrable holomorphic 1-forms $\varphi$ on $X$ such that $\varphi|_{z_0}=c_{\beta}(z_0)dw|_{z_0} $.

\begin{proof} [\bf Proof of $(i) \Longrightarrow (iii)$] The Bergman kernel $\kappa_X $ satisfies the extremal property, namely at any $z_0 \in X$, it holds that
\begin{equation} \label{sup}
\kappa_X(z_0)=\sup \left \{ \left. h(z_0) \wedge \overline {h (z_0)} \,\, \right | \,\,h \text{ is a holomorphic 1-form and}\left \| h \right \|_{L^2}=1  \right\}.
\end{equation}
Suppose $\pi K(z_0)=c^2_{\beta}(z_0)$.  Then
 the form
\begin{equation} \label{F}
F:=\frac{  {\sqrt{\pi}K (s , \overline{z_0})}} {\sqrt{K(z_0)}} dt,
\end{equation}
 belongs to $\mathcal H$, with $L^2$-norm being $\sqrt{\pi}$. If there exists another form $\varphi \in \mathcal H$ with $||\varphi||_{L^2} < \sqrt {\pi}$, then due to the extremal property \eqref{sup}, $\pi K(z_0) > c^2_{\beta}(z_0)$, which is a contradiction.
 So the form $F$ locally given by \eqref{F} attains the minimal $L^2$-norm among all forms in $\mathcal H$. By general principles from Hilbert space theory, we know that the element of minimal norm is unique, so there exists a unique form $F \in \mathcal H$ with $||F||_{L^2}\leq \sqrt \pi$. Moreover, such unique $F$ given by \eqref{F} satisfies $||F||_{L^2}\ =\sqrt \pi$. 

\end{proof}

\begin{proof} [\bf Proof of $(iii) \Longrightarrow (i)$] Consider the form  
$$ 
h:=\frac{K(s , \overline{z_0}) c_{\beta}(z_0)}{K(z_0)} dt.
$$ 
We know that $h\in \mathcal H$ with $L^2$-norm being ${c_{\beta}(z_0)}/\sqrt{K(z_0)}$, which is less than $\sqrt{\pi}$ by the inequality part \eqref{leq} of Suita's conjecture. Condition (iii) says that for the unique form $h$ in $\mathcal H$ with $L^2$-norm not greater than $\sqrt \pi$, it should satisfy $||h||_{L^2}\ =\sqrt \pi$, which yields that $\pi K(z_0)=c^2_{\beta}(z_0)$. 

\end{proof}

Next, we prove that in Theorem \ref{main}, Conditions $(i)$ and $(iv)$ are equivalent.

\begin{proof} [\bf Proof of $(iv) \Longrightarrow (i)$] Suppose $\log K_{\tau}(z_0)$ is harmonic in $\tau$. For each fixed $\tau\in \mathbb C$, since the definition of ${X_\tau}$ in \eqref{X_tau} depends only on $\Real \tau$, so does its Bergman kernel $ \kappa_{X_\tau}(z_0)$. Therefore, $\log K_{\tau}(z_0)$ is linear in $\Real \tau$, and we write 
$$
\log K_{\tau}(z_0) = C \Real \tau + \log K(z_0),
$$
since $K_{0}(z_0)$ is just $K (z_0)$. By \eqref{X_tau} again,  the Green's function on ${X_\tau}$ with a pole at $z_0$ is $G(z, z_0)-\Real \tau$. It then follows that the logarithmic capacity $c_{\beta, \tau}$ on ${X_\tau}$ defined by \eqref{cap} satisfies 
\begin{equation} \label{c_tau}
c_{\beta, \tau}(z_0)= c_{\beta }(z_0)\cdot e^{-\Real \tau}.
\end{equation} 
As $\Real \tau \to -\infty$, ${X_\tau}$ becomes a simply connected surface, on which the equality condition in Suita's conjecture holds, namely
\begin{equation} \label{=_tau}
\pi K_\tau(z_0)= c^2_{\beta, \tau}(z_0).
\end{equation} 
Therefore, for all $\tau$ such that $\Real \tau \ll 0$, 
$$
\frac{\pi K(z_0)}  {c^2_{\beta}(z_0)}=e^{(C-2)\Real \tau},
$$
which forces $C=2$ since the above left hand side is independent of $\tau$. So we get (i).

\end{proof}

To prove the converse direction, we need an elementary but crucial {\bf fact}: {\it a convex function of one real variable defined on the negative half axis must be monotonically non-decreasing if it is bounded from above at $-\infty$}. This fact can be proved using the Taylor  expansion.

\begin{proof} [\bf Proof of $(i) \Longrightarrow (iv)$] For the family of surfaces $X_{\tau}$ defined by  \eqref{X_tau}, by Theorem \ref{M-Y}, $\log K_{\tau}(z_0)$ is subharmonic in $\tau$.  
 In this case, $\log K_{\tau}(z_0)$ is convex in $\Real \tau$, since $ K_{\tau}(z_0)$ depends only on $\Real \tau$.
  Then, $\log K_{\tau}(z_0)+ 2\Real \tau$, which is also convex, is asymptotic to $2\log c_\beta(z_0)-\log \pi$ as $\Real \tau\to -\infty$, when ${X_\tau}$ becomes simply connected (see \cite{BL}). Therefore, $\log K_{\tau}(z_0)+ 2\Real \tau$ is bounded from above at $-\infty$, and by the above elementary fact, it is non-decreasing. In particular, for any $\Real \tau\leq 0$, $$\log K_{\tau}(z_0)+ 2\Real \tau\leq \log K(z_0).$$ 
  If $\pi K(z_0)=c^2_{\beta}(z_0)$, then
  \begin{equation} \label{K_tau}
\log K_{ \tau}(z_0)+ 2\Real \tau = \log K(z_0),
\end{equation}
so $\log K_{\tau}(z_0)$ is harmonic in $ \tau$, and in fact linear in $\Real \tau$. 

\end{proof}

 See Corollary \ref{v} for the equivalent (but weaker) versions of Conditions (iii) and (iv) in Theorem \ref{main}.

 \section{Proofs of Key Lemma \ref {off} and Theorem \ref{main}}  \label{Sec3}

 Let $X$ be an open Riemann surface admitting a Green's function $G(z, z_0)$ with a pole $z_0\in X$. For given $\tau\in \mathbb C$, let $X_\tau$ be the sub-surface defined by \eqref{X_tau}, and let $\kappa_{X_\tau}$ be the Bergman kernel for holomorphic (1,0) forms on $X_\tau$.

 \begin{klem} \label{off} On an open Riemann surface $X$ admitting a Green's function, if $\pi K(z_0)=c^2_{\beta}(z_0)$, then
$$
 \left. \kappa_{X}(s, \overline{z_0}) \right|_{X_\tau} = \kappa_{X_\tau}(s, \overline{z_0})  \cdot e^{\tau+\bar{\tau}},
 $$
 for all $\tau$ such that $\Real \tau \leq 0$ and the Bergman kernel $\kappa_{X_\tau}(s, \overline{z_0})$ has no zero in $X_\tau$.
 \end{klem}

\begin{proof}  

 Step {\bf 1}) In a local coordinate system $(U\times \overline{V},(t, \overline{w} ))$ for $X_\tau\times\overline{X_\tau}$ such that $w(z_0)=0$,
we define
$$
 K_\tau(s, \overline{z}) dt\otimes d\overline{w} := \kappa_{X_\tau}(s, \overline{z})|_{U\times \overline{V} }, \quad 
K_\tau(z)|dw|^2 := \kappa_{X_\tau}(z)|_{V }.
$$
 In particular, $\kappa_{X}= \kappa_{X_0}$.

  We {\bf claim that}:  on each $X_{\tau}$,
$$ 
F_{\tau}:=\frac{e^{\tau+\bar \tau} {\sqrt{\pi}K_{\tau}(s, \overline{z_0})}} {\sqrt{K(z_0)}} dt
$$ 
is the unique holomorphic 1-form that satisfies $F_{\tau}|_{z_0}=   c_{\beta}(z_0) dw|_{z_0} $and
$$||F_{\tau}||_{L^2(X_{\tau})}\leq \sqrt \pi \cdot e^{\Real \tau}.$$

In fact, by \eqref{K_tau}, the form
$$
T_{\tau}:=\frac{{\sqrt{\pi}K_{\tau}(s, \overline{z_0})}e^{\Real \tau}}{{\sqrt{K(z_0)}}} dt=\frac{{\sqrt{\pi}K_{\tau}(s, \overline{z_0})} }{ {\sqrt{K_{{\tau}}(z_0) }}} dt
$$
 has $L^2$ norm on $X_{\tau}$ equaling $\sqrt{\pi}$. Moreover, by \eqref{=_tau} and \eqref{c_tau} from the proof of the equivalence between (i) and (iv), 
 $T_{\tau}|_{z_0}=c_{ \beta, \tau}(z_0)dt |_{z_0} =c_{ \beta}(z_0) dw  |_{z_0} \cdot e^{-{\Real \tau}}$.
So, we define $F_{\tau}:=T_{\tau}\cdot e^{\Real \tau}$, and by the equivalence between (i) and (iii), the uniqueness in the claim follows.

\medskip

 Step {\bf 2}) 
We may first assume that the boundary $\partial X_{\tau}$ consists of a finite number of $C^2$ smooth closed curves and $\partial X_{\tau}$ varies $C^2$ smoothly with $\tau$ in $\mathbb C$. Then, by \cite[Theorem 3.2]{MY},
\begin{align*}
\frac{\partial^2 K_{\tau}(z_0)}{\partial \tau \partial\bar\tau}= &\int \limits_{\partial X_{\tau}} {k_2(\tau, s)} \left( |\mathcal L_{\tau}(s, \overline{z_0})|^2+| K_{\tau}(s, \overline{z_0} ) |^2\right)d\sigma_s\\
&+\iint \limits_{X_{\tau}} \Big(\left|\frac{\partial \mathcal L_{\tau}(s, \overline{z_0})}{\partial \bar \tau}\right|^2+\left| \frac{\partial K_{\tau}(s, \overline{z_0})}{\partial \bar \tau}\right|^2\Big)dv_s,
\end{align*}
where $dv_s={|ds|^2}/{2}$ is the Lebesgue $\mathbb R^2$ measure. In fact, the definition of $\mathcal X$ in \eqref{total} implies that $k_2(\tau, s)=0$ on $\partial \mathcal X:=\coprod _{\tau\in \mathbb C} {\partial X_{\tau}}.$ Therefore, by \eqref{K_tau},
\begin{equation} \label{less} 
\iint _{X_{\tau}} \left| \frac{\partial K_{\tau}(s, \overline{z_0})}{\partial \bar \tau}\right|^2 dv_s \leq \frac{\partial^2 K_{\tau}(z_0)}{\partial \tau \partial\bar\tau}  =K(z_0)\cdot e^{-\tau-\bar\tau}.
\end{equation} 
Consider on $X_{\tau}$ the square integrable 1-form $$B_{\tau}:=\frac{-\partial K_{\tau}(s, \overline{z_0})}{\partial \bar \tau}  \frac{\sqrt{\pi} e^{\tau+\bar\tau}}{\sqrt{K(z_0)}} dt.$$ It is holomorphic in $s$, and by \eqref{less} satisfies
$$
||B_{\tau}||^2_{L^2(X_{\tau})}\leq  \pi  e^{ \tau+\bar\tau }.
$$
Moreover, 
$$
B_{\tau}|_{z_0}=\frac{-\partial K_{\tau}(z_0)}{\partial \bar \tau}  \frac{\sqrt{\pi} e^{\tau+\bar\tau}}{\sqrt{K(z_0)}} dt|_{z_0}=\sqrt{\pi K(z_0)}dw|_{z_0}=c_\beta(z_0)dw|_{z_0}.
$$
By the uniqueness in Step 1), we know that $B_{\tau}\equiv F_{\tau}$ on  $X_{\tau}$, i.e., 
\begin{equation} \label{PDE} 
\frac{\partial K_{\tau}(s, \overline{z_0})}{\partial \bar \tau}\equiv-K_{\tau}(s, \overline{z_0}).
\end{equation} 
Since by assumption $K_{\tau}(s, \overline{z_0})\neq 0$ for all $s \in X_{\tau} \subset X$, 
\begin{align*}
 \frac{\partial^2 \log |K_{\tau}(s, \overline{z_0})|^2}{\partial \tau \partial\bar\tau} =&\frac{\partial}{\partial \tau} \left \{ | K_{\tau}(s, \overline{z_0})|^{-2}   \left(  -K_{\tau}(s, \overline{z_0})   \overline {K_{\tau}(s, \overline{z_0})}+ K_{\tau}(s, \overline{z_0}) \frac{\partial \overline {K_{\tau}(s, \overline{z_0})}}{ \partial\bar\tau}  \right)\right \}\\
=& - { {\overline{ K_{\tau}(s, \overline{z_0})}^{-2}}}   \frac{\partial \overline{ K_{\tau}(s, \overline{z_0})}}{\partial \tau}  \frac{\partial \overline{ K_{\tau}(s, \overline{z_0})}}{\partial \bar{\tau}} +  {\overline{ K_{\tau}(s, \overline{z_0})}}^{-1}    \frac{\partial}{ \partial\bar {\tau}}\frac{\partial  \overline {K_{\tau}(s, \overline{z_0})}}{ \partial \tau} \\
=&-{ {\overline{ K_{\tau}(s, \overline{z_0})}^{-2}}}  (- \overline{ K_{\tau}(s, \overline{z_0})})  \frac{\partial \overline{ K_{\tau}(s, \overline{z_0})}}{\partial \bar{\tau}} +  {\overline{ K_{\tau}(s, \overline{z_0})}}^{-1}   \frac{\partial (- \overline {K_{\tau}(s, \overline{z_0})})}{ \partial\bar {\tau}}=0.
\end{align*} 
This implies that $\log |K_{\tau}(s, \overline{z_0})|$ is harmonic in $\tau$ and moreover linear in $\Real \tau$. Since $K_{0}(s, \overline{z_0})$ is just $K(s, \overline{z_0})$, we write 
$$
\log |K_{\tau}(s, \overline{z_0})|=C(s, \overline{z_0})\cdot \Real \tau+\log |K(s, \overline{z_0})|,
$$
 where $C(s, \overline{z_0})$ is a real-valued function of $s$ and $z_0$. 
Therefore,
$$
\left|\frac{K(s, \overline{z_0})}{K_{\tau}(s, \overline{z_0})}\right|=e^{ -\Real \tau \cdot C(s, \overline{z_0})},
$$ 
and thus there is a $\theta \in \mathbb R$ such that $$\left. K(s, \overline{z_0}) \right|_{X_\tau} =K_{\tau}(s, \overline{z_0}) \cdot e^{ -\Real \tau \cdot C(s, \overline{z_0})}\cdot e^{i \theta}.$$ By letting $\tau$ be $0$, we know $\theta =0$. By \eqref{PDE}, $C(s, \overline{z_0})\equiv -2$. Finally, for a general surface $X$, we may approximate it by an increasing sequence of smoothly bordered  ones whose union is $X$, since the Bergman kernel, Green's function and logarithmic capacity converge locally uniformly. 
Without loss of generality, it suffices to assume that $X$ has smooth boundary.

\end{proof}

 Let $X$ be an open Riemann surface admitting a Green's function $G(z, z_0)$ with a pole at $z_0\in X$. Let $w$ be a local coordinate on a neighborhood $V $ of $z_0 $ satisfying $w(z_0) = 0$.
 Since $G(z, z_0)-\log|w(z)|$ is harmonic, there exists a holomorphic function $f_1$ on $V $ such that $\Real f_1(z) = G(z, z_0)|_{V } -\log|w(z)|$. Define the local (injective) holomorphic function
\begin{equation} \label{f} 
f(z):=w(z) e^{f_1(z)- i \Imagine{f_1(z_0)}}.
\end{equation} 
 Then, $f (z_0)=0$, $f^{\prime}(z_0)>0$ and  $G(z, z_0)|_{V }= \log|f(z)| $.

 \medskip
 
The following Theorem \ref{lemma}, which corresponds to Guan and Zhou's Proposition 4.21/Corollary 5.4 in \cite{GZ3}, played a decisive role in their proof of the equality part in Suita's conjecture.

\begin{theorem}  [Proposition 4.21/Corollary 5.4 in \cite{GZ3}] \label{lemma} On an open Riemann surface $X$ admitting a Green's function, if $\pi K(z_0)=c^2_{\beta}(z_0)$, then  there exists a holomorphic 1-form $F$ on $X$ such that $F|_{V} = df$,
where $f$ is defined by \eqref{f}.
\end{theorem}

However, Guan and Zhou's proof of Theorem \ref{lemma} is highly non-trivial and heavily relied on the Main Theorems 1 and 2 developed in \cite{GZ3}. An intricate selection of a special function $C_A(t)$ became the main difficulty and one of the key points in \cite{GZ2, GZ3}.  Their idea was to extend $df$ by using the $L^2$ extension theorem of a very general version. We will give a purely variational proof of Theorem \ref{lemma} without using the $L^2$ extension theorem, as the global form $F$ defined by \eqref{F} already becomes a natural candidate.

\begin{proof} [\bf Proof of Theorem \ref{lemma} (without $L^2$ extension theorem)]

Take $\Real \tau  \ll 0$ such that $X_{\tau}$ defined by \eqref{X_tau} is simply connected. 
Then, $f$ defined by \eqref{f} is a biholomorphism from $X_{\tau}$ to $\mathbb D_r$, where $\mathbb D_r$ is a disk centred at $0\in \mathbb C$ with radius $r:=\exp (\Real \tau)$.
By the transformation rule of the Bergman kernel, we know that 
$$
K_{\tau}(s, \overline{z_0})=f^{\prime} (s) \cdot \overline{f ^{\prime}(z_0) }\cdot K_{\mathbb D_r}\left(f (s), \bar 0\right) = \frac{f ^{\prime}(t) \cdot f^{\prime} (z_0)} { \pi r^2}.
$$
 Letting $s=z_0$, we get $f ^{\prime}(z_0)=r\sqrt{\pi K_{\tau}(z_0)}$, so for all $s \in X_{\tau}$,
 \begin{equation} \label{df}
df  =  \frac{r \sqrt {\pi} K_{\tau}(s, \overline{z_0})}{\sqrt{K_{\tau}(z_0)}}  dt.
\end{equation}  
By Key Lemma \ref{off}, since $K_{\tau}(s, \overline{z_0})$ is zero-free, the form $F$ defined by \eqref{F} satisfies 
\begin{equation*} 
F |_{X_{\tau}}= \left. \frac{ {\sqrt{\pi}K (s, \overline{z_0}) }} {\sqrt{K(z_0)}} dt  \right| _{X_{\tau}}=\frac{ {\sqrt{\pi}K_{\tau} (s, \overline{z_0}) }\cdot e^{ 2\Real \tau}} {\sqrt{K_{\tau}(z_0)} \cdot e^{ \Real \tau}} dt=df,
\end{equation*}  
which completes the proof of  Theorem \ref{lemma}.

\end {proof}

The existence of the holomorphic 1-form $F$ in Theorem \ref{lemma} guarantees the existence of a global holomorphic function whose absolute value is equal to $\exp G(z, z_0)$, as demonstrated in \cite [pages 1195--1196] {GZ3} by Guan and Zhou, whose construction was to lift up local properties to the universal covering space and then descend them back. We will give a straightforward argument of such an implication by simply applying the uniqueness property of real-analytic functions, and consequently complete the proof of Theorem \ref{main}.

\begin{proof} [\bf Proof of Theorem \ref{main} (without $L^2$ extension theorem)] If $\pi K(z_0) = c^2_{\beta}(z_0)$, then by Theorem \ref{lemma} there exists a holomorphic 1-form $F$ on $X$  such that $F|_{V} = df$,
where $f$ is a holomorphic function on a neighborhood $V $  of $z_0 $ such that  $ |f(z)| =\exp G(z, z_0)|_{V } $.
So, for $ z \in V$,
\begin{equation} \label {same} 
 \left. \frac{F}{2 {\partial G(z, z_0)} } \right|_{V } =   \frac{df}{ {\partial \log (f(z) \overline{f(z)})}  } = \frac{ \partial f}{\frac{ \partial  f(z)  }{ f(z)}} = f(z). 
\end{equation} 
Both $\left | \frac{F}{2  \partial G(z, z_0)}  \right| $ and $\exp G(z, z_0)$ are defined on $X $. So by \eqref{same} and the uniqueness  of real-analytic functions, 
it holds on $X$ that 
$$
\exp G(z, z_0)=   \left | \frac{F}{2  {\partial G(z, z_0)} } \right|.
$$
As $G(\cdot , z_0)$ is harmonic, $   (2  {\partial G(z, z_0)} )^{-1} F $ is a holomorphic function on $X$. Then, by \cite[Lemma A.2] {DTZ} (see also \cite{Min}), the surface is necessarily biholomorphic to a disc possibly less a relatively closed polar set. Thus, we have proved the implication \eqref{circle}, and particularly (i) implying (ii). The converse direction is trivial as the polar part is negligible, so the proof is complete.

\end{proof}

From the above proof, we see that the converse of Theorem \ref{lemma} holds true as well. Moreover, we find that Conditions (iii) and (iv) in Theorem \ref{main} are in fact equivalent to their weaker versions.

\begin{cor}   \label{v} Under the same assumptions as in Theorem \ref{main}, either of the Conditions (i)--(iv) is equivalent to either  of the followings:
  \begin{enumerate}
  
  \item [$(iii^{\prime})$] there exists a unique holomorphic 1-form $F $ on $X$ with $F |_{z_0}=c_\beta(z_0)dw|_{z_0} $ and $||F||_{L^2}\leq \sqrt \pi$;
  
      \item [$(iv^{\prime})$]  $ \log K_{\tau}(z_0)  + 2\Real \tau = \log K (z_0) $ for some $\tau$ such that  $\Real \tau<0$;
      
\item [$(v)$] there exists a holomorphic 1-form $F$ on $X$ such that $F|_{V} = df$,
where $f$ is defined by \eqref{f}.
  \end{enumerate}

\end{cor} 

 \begin{proof}  
{\bf Part $(v)$}. Theorem \ref{lemma}  shows that (i) $\Longrightarrow$ (v).
 The construction in the above proof of Theorem \ref{main} shows that (v) $\Longrightarrow$ (ii).

\medskip

 {\bf Part $(iii^{\prime})$}. Guan and Zhou in \cite [Proposition 4.21/Corollary 5.4] {GZ3} showed that (iii$^{\prime}$)   implies (v), which is further equivalent to (iii), by Part (v) of our corollary. Conversely,  (iii)   implies (iii$^{\prime}$), as the unique form in $ \mathcal H$ in general may have its $L^2$ norm strictly less than $\sqrt \pi$.
 
 \medskip

{\bf Part $(iv^{\prime})$}.  The harmonicity of $ \log K_{\tau}(z_0)$ in $\tau$ implies $ \log K_{\tau}(z_0)  + 2\Real \tau = \log K (z_0) $ for all $\tau$ such that  $\Real \tau<0$. Conversely, (iv$^{\prime}$) $\Longrightarrow$ (iv) due to the convexity and monotonicity of $ \log K_{\tau}(z_0)  + 2\Real \tau $ in $\Real  \tau $.
    
\end{proof}

\section{Results on the Bergman metric of constant curvature} \label{Sec4}
Let $X$ be a Riemann surface, 
and let $\kappa_X$ be the {\it Bergman kernel} for holomorphic (1,0) forms on $X$.
In a local coordinate system $(V\times \overline{U},(w, \overline{t} ))$ for $X\times\overline{X}$,
we define
$$
 K(z, \overline{s}) dw\otimes d\overline{t} := \kappa_X(z, \overline{s})|_{V\times \overline{U} },
$$ 
and
\begin{equation} \label{KLocal}
K(z)|dw|^2 := \kappa_X(z)|_{V }, \quad 
K(s)|dt|^2 := \kappa_X(s)|_{U }.
\end{equation}
The Hermitian form
\begin{equation} \label{metric}
 g (z)dw\otimes d\overline{w}:=\frac{\partial^2\log K (z)}{\partial w\partial\overline{w}}dw\otimes d\overline{w}.
\end{equation} 
is called the {\it Bergman metric} on $X$ if it is positive everywhere. 

\medskip 

Note that \eqref{metric} is independent of the choice of the local coordinate system.  If $( V^{1} \times \overline{U^{1}},(w^{1}, \overline{t^{1}}))$ is another coordinate system, then in $(V\times\overline{U})\cap(V^{1}\times\overline{U^{1}})$,
\begin{equation}\label{trk}
K (z,\overline{s})=K^{1} (z,\overline{s})  \frac{\partial w^{1}}{\partial w} \overline{ \frac{\partial t^{1}}{\partial t} }, \quad g (z)=  {g}^{1}( {z }) \left | \frac{\partial w^{1}}{\partial w} \right|^2,
\end{equation}
where $g^{1}(z )= ( \log K^1 (z) )_{w^{1} \overline{w^{1}}}$. By the transformation formula of the Bergman kernel (under biholomorphisms), every biholomorphism is an isometry with respect to the Bergman metric. 

\medskip 

Given a point $ {p}\in {X}$, define the analytic variety
\begin{equation} \label{zero set}
Z_p:=\{z\in X: \kappa_X(z, \overline{p})=0\}.
\end{equation}
 The transformation formulae (\ref{trk})  shows that this set is well-defined. Therefore, as Riemann surfaces $X \setminus Z_{p}$ and $X$ have the same Bergman kernel and Bergman metric.

\medskip 

For any fixed $s\in X$, consider 
 \begin{equation} \label{dia}
z \longmapsto  \log \frac{K (z)   K (s)}{ |K (z,\overline{s})|^2},
\end{equation}
which is a well-defined real-valued real-analytic function on $X \setminus Z_{s}$. The right hand side of \eqref{dia}, denoted by $\Phi_s(z)$, is a K\"{a}hler potential
for the Bergman metric $g= (\Phi_{s})_{w\bar w}$ under the local coordinate system $(V , w )$.
Locally, $\Phi_s(z)$ coincides with Calabi's diastasis, so we call $\Phi_{s}(z)$ the Bergman-Calabi diastasis relative to $s$.

\begin{definition}

Let $X$ be a Riemann surface which possesses the Bergman metric. For any $p \in X$, 
let $t$ be a local coordinate such that $t(p) = 0$. The {\it Bergman representative coordinate} $T$  relative to $p$ is defined by 
\begin{equation} \label {rep}
z \longmapsto T(z):={g}^{-1}(p)\Big\{ K^{-1}(z,\overline{p}) \frac{\partial}{\partial \overline{t}}\Big|_{s=p}  K (z,\overline{s})-\frac{\partial}{\partial \overline{t}}\Big|_{s=p}\log K (s)\Big\}.
\end{equation}
\end{definition}

The above construction is independent of the choice of a coordinate system $(V, w)$, but it depends on the choice of local coordinate system $(U, t)$. 
In fact,  if we change the local coordinate system $(U, t)$ to another coordinate system $(   U^1, t^1)$ around  $p$, then 
\begin{equation} \label{trT}
T(z) =   {T^1(z)}   \frac{\partial t}{\partial t^{1}}   (p) , \quad z\in X \setminus Z_p,
\end {equation}
where $Z_p$ is well-defined by \eqref{zero set}. 
Since all holomorphic vector bundles over $X$ (which we can  assume to be non-compact after a possible removal of a single point) are holomorphically trivial, one may fix some linear trivialization so that $T$ becomes a holomorphic function on $X \setminus Z_p$. Also,  $ T^{\prime} (p)= 1$.

\medskip

   Following the argument of \cite{DW}, we   get

     \begin{pro} \label{zero} 
     Let $X$ be a Riemann surface on which the Gaussian curvature of the Bergman metric  is identically $-2$. For any $p \in X$, let $ T $ be the Bergman representative coordinate defined by \eqref{rep}.
     Then,
      the followings  hold true:
 \begin{enumerate}

\item [1)] $T$ is a holomorphic isometry with respect to the Bergman metric from $X$ to a disc $\mathbb D_r :=\{ \eta \in \mathbb C :    |\eta |^2   < {2 g^{-1} (p)}  \}$;

\item [2)] $X$ is hyperbolic, namely, an open surface admiting a negative non-constant subharmonic function;

\item [3)] the Bergman kernel $\kappa_X(z, \overline{p})$ has no zero set, i.e., the surface $X$ is Lu Qi-Keng;

\item [4)]  the Bergman-Calabi diastasis relative to $p$  can be written as

\begin{equation}  \label{Lu_g}
\Phi_{p}( z) =   -2  \log      \left (1 - \frac{1}{2}   |T(z)|^2   g (p)   \right ), \quad z \in  X.
\end{equation}

 \end{enumerate}

\end{pro}

Notice that by \eqref{trk} and  \eqref{trT},  $|T(z)|^2 g (p)$ is independent of the choice of the local coordinate system around  $p$. Thus the right hand side of \eqref{Lu_g} is indeed a function on $X$.

\begin{proof}  [\bf Proof of Proposition \ref {zero} ]

{\bf 1)}
Choose some local coordinate system $( U , t )$ for $X $ such that $t(p) = 0$. For simplicity, by \eqref{KLocal} and \eqref{metric}, let $K:=K(s)$ and $g:=g(s)$.
The Gaussian curvature of the Bergman metric on $X$ is defined as
 $$
 R_{X} :=  2 g ^{-2}   R,
 $$
 where the curvature tensor is
 given by
$$
R =-\frac{\partial^4}{\partial t^2  \partial  \overline{t} ^2} \log K  +    g^{ -1}  \frac{\partial^3 }{\partial t^2   \partial \overline {t }} \log K  \frac{\partial^3 }{\partial\overline {t}^2  \partial   {t}  }  \log K.
 $$
If the Gaussian curvature is identically $-2$, then
$$
R =- g^2,
$$
which implies that
\begin{equation} \label{const}
\frac{\partial^4}{\partial t^2   \partial \overline {t} ^2 } \log K=  g^{ -1}  \frac{\partial^3 }{\partial  t ^2 \partial \overline {t }} \log K  \frac{\partial^3 }{\partial  \overline {t}^2 \partial  {t}  }  \log K +   g.
\end{equation}

Consider the test function 
 $$
\phi(s):= -2 \log  \left (1 - \frac{1}{2} |T(s)|^2 g(p)  \right ) .
 $$
Then, $\phi$ induces a K\"{a}hler  metric whose Gaussian curvature is also identically $-2$, with a similar identity as \eqref{const}.
Direct computations (see \cite{Lu, DW}) yield that
$$
\left. \frac{\partial^2 \log K }{\partial t  \partial \overline {t  }} \right|_{s=p}=\left. \frac{\partial^2 \phi(s)}{\partial t  \partial \overline {t  }} \right|_{s=p}=g  (p),
$$
$$
\left. \frac{\partial^3 \log K }{\partial^2 t     \partial \overline {t } } \right|_{s=p}=\left. \frac{\partial^2 \phi(s)}{\partial ^2 t  \partial \overline {t } } \right|_{s=p}=0.
$$
 The partial derivatives of order 4 can be computed directly from \eqref{const}; furthermore, the partial derivatives of higher order can be successively computed by taking all possible successive derivatives of \eqref{const}. As a result, they all vanish at $s=p$.
By the uniqueness of the Taylor expansion, it holds that
\begin{equation} \label{K(z, z)}
K (s)=  \left (1 - \frac{1}{2}  |T(s)|^2 g(p) \right )^{-2} e^{F(T(s))+\overline{F(T(s))}}, \quad s\in U,
\end{equation}
for some holomorphic function $F$.

Let  $X^{\prime}:= \{z \in X \setminus Z_p : T(z) \in \mathbb D_r \}$ be the set of points in $X \setminus Z_p$ that are mapped into the disc $\mathbb D_r$.  In particular, $U  \subset X^{\prime}$.

By \eqref{K(z, z)} and the theory of power series, one may duplicate the variable with its conjugate so that  the full Bergman kernel can be complex analytically continued as
\begin{equation}  \label{K(z, z_0)}
K (s, \overline{s_0})= \left (1 - \frac{1}{2}    T(s)   g(p)  \overline {T(s_0)}   \right )^{ -2}   e^{f(T(s))+\overline{f(T({s_0}))}}, \quad (s, s_0) \in U \times \overline{U}.
\end{equation}
Then for any $s_0\in U$, by definition,
 \begin{align*}
\Phi_{s_0}( s)&= \log \frac{\left (1 - \frac{1}{2} |T(s)|^2 g(p) \right )^{ {-2} } e^{f(T(s))+\overline{f(T(s))}}   \left (1 - \frac{1}{2} |T(s_0)|^2 g(p) \right )^{-2} e^{f(T(s_0))+\overline{f(T(s_0))}}  }{  \left |1 - \frac{1}{2}     T(s)  g(p)   \overline {T(s_0)}   \right |^{ -4 } |e^{f(T(s))+\overline{f(T({s_0}))}}  |^2}\\
& =  {-2}   \log  \left[    \left (1 - \frac{1}{2} |T(s)|^2 g(p) \right )  \left (1 - \frac{1}{2} |T(s_0)|^2 g(p) \right )  \left  | 1 - \frac{1}{2}     T(s)   g(p)  \overline {T(s_0)}   \right  |^{ -2}  \right  ], \quad s\in U ,
 \end{align*}
which yields that
$$
\Phi_{p}(s_0)= \Phi_{s_0}(p) = {-2} \log  \left (1 - \frac{1}{2} |T(s_0)|^2 g(p) \right ).
$$
 On the other hand, the Bergman-Calabi diastasis $\Phi_{p}(s)$ relative to $p$ is defined on $X \setminus Z_p$ and thus on $X^{\prime}$, where ${-2} \log  \left (1 - \frac{1}{2} |T(s )|^2 g(p) \right )$ can be defined. Since these two real-analytic  functions coincide on  $U$, they are identical to each other on $X^{\prime}$. That is, 
\begin{equation} \label{on X prime}
\Phi_{p}(s)=  {-2} \log  \left (1 - \frac{1}{2} |T(s )|^2 g(p) \right ), \quad s \in X^{\prime}.
\end{equation}

 We {\bf claim that}:  no point in  $\Omega \setminus Z_p$ is mapped outside the disc $\mathbb D_r$ by $T$. 

If not, suppose there exists some point $q \in  X \setminus Z_p$ that is mapped to $\{ \eta \in \mathbb C:   |\eta|^2 \geq  {2 g^{-1} (p)}   \}$. Choose some point $q_0 \in X^{\prime}$.  Since $X \setminus A_p$ is path-connected, one can choose a path $\gamma$ that connects $q_0$ and $q$. Suppose under $T$ the image of $\gamma$   intersects $\partial  \mathbb D_r$ firstly at some point $T(q_1)$.

 Along the path  $\gamma$ take a sequence of points $(q_l)_{l \in \mathbb N} \subset X^{\prime}$ such that $q_l  \to q_1$.
 Then by \eqref{on X prime},
$$
\Phi_{p}(q_l )=  {-2} \log  \left (1 - \frac{1}{2} |T(q_l)|^2 g(p) \right ).
$$
Here, as  $q_l  \to q_1$,  the left hand side is finite but  the right hand side  blows up to infinity.  This  is  a contradiction, so we have thus proved our claim, which says that $X^{\prime}=  X \setminus Z_p $.

Therefore, by  our claim we know that \eqref{on X prime} in fact holds on $X \setminus Z_p $. Since $T$ maps $X \setminus Z_p$ holomorphically  to the disc $\mathbb D_r$ and satisfies 
\begin{equation} \label{bound}
 |T |^2<  {2 g^{-1} (p)}  ,
\end{equation}
by the Riemann removable singularity theorem, $T$ extends across the analytic variety $Z_p$ to the whole surface $X$ with $|T |^2 \leq {2 g^{-1} (p)} .$
  The maximum modulus principle then yields  that \eqref{bound} in fact holds  on $X$, namely, 
$T$ mapping $X$ holomorphically to the disc  $\mathbb D_r$.
 The isometry part will be proved in Part 4).

\medskip

{\bf 2)} By 1), $T$ is a non-constant bounded holomorphic function on $X$, and satisfies \eqref{bound}. Then, $2 \log |T| +  \log {\frac{g(p)}{2}} $ becomes a negative non-constant subharmonic function on $X$, which is necessarily hyperbolic.

\medskip

{\bf 3)}  Suppose
$Z_{p} \neq \emptyset$, i.e., $X$ is not Lu Qi-Keng.
Then take some point $q \in Z_{p}$ and take a sequence of points $(s_j)_{j\in \mathbb N} \subset X  \setminus Z_p$ such that $s_j \to q$.
By \eqref{on X prime},
$$
\Phi_{p}( s_j) =  {-2}   \log      \left (1 - \frac{1}{2} g(p) |T(s_j)|^2 \right ).
$$
Letting $z_j  \to q$, we see that the above left hand side blows up to infinity, but  the right hand side is finite. This  is a contradiction, so $Z_{p} = \emptyset$.

\medskip

{\bf 4)}
By 2) and \eqref{on X prime}, we know that the formula \eqref{Lu_g} holds true on $X$. So  the explicit formula of the Bergman metric can be further computed as
 \begin{equation} \label{V}
g(s) |dt |^2=  | T^{\prime}(s)|^2 g (p) \left (1 - \frac{1}{2}    |T(s)|^2   g (p)   \right )^{-2}  |dt |^2, 
 \end{equation}
which shows that $T$ is an isometry.

   \end {proof}

      \section{Proofs of Theorems \ref{2nd}, \ref{3rd}, \ref{lower}, \ref{4th} and \ref{boundary point}}  \label{Sec5}

  \begin{proof} [\bf  Proof of Theorem \ref{3rd}] 

By definition, the Gaussian curvature of the Bergman kernel $K(z)|dw|^2$ 
is 
\begin{equation} \label{curvK}
-4   \frac{{( \log \sqrt{K (z)})}_{ w \overline{w}} }{K(z)} = -2 \frac{g(z)}{K(z)}.
\end{equation} 
Thus, Condition 2 in Theorem \ref{3rd} is equivalent to the existence of $p \in X$ such that $g(p) \geq 2 \pi  K(p)$. This further implies Condition 1 by \eqref{leq} in Theorem \ref{main}, since $X$ is hyperbolic
by Proposition \ref {zero}, Part 2).
Therefore, it suffices to prove the theorem by assuming Condition 1.

It is well known that $c_\beta^2  \geq c_B^2$, where $c_B$ is the analytic capacity locally defined as 
 $$
 c_B(p):=\sup \left \{ \left|\frac{dh}{dw}({p})\right| \, : h({p})= 0, \, \,h\text{ is holomorphic and} \, \, |h|\leq1 \right\},
 $$
with respect to a fixed local coordinate $w$ near $p \in X$ such that $w(p)=0$.
 Consider the holomorphic function $H(z):=\sqrt{\frac{g(p)}{2}} T(z)$ on $X$.  Since $H({p})= 0$, $  T^{\prime}(p) =1$ and   $|H(z)| \leq1$ on $X$ by Proposition \ref{zero},  we get
\begin{equation} \label{ineq}
c_B^2(p) \geq   |T^{\prime}(p)|^2  \frac{g(p)}{2} =   \frac{g(p)}{2}. 
\end{equation}

Our assumption and \eqref{ineq} imply that $c_B^2(p) = c^2_{\beta}(p)$.
Using \cite[Theorem A.1] {DTZ} (see also \cite{Min}), we know that $X$ is biholomorphic to a disc possibly less a relatively closed polar set.

\end{proof}

We will prove Theorem \ref{2nd} by using our Theorem \ref{3rd}.

\begin{proof} [\bf  Proof of Theorem \ref{2nd}]

If $X$ is biholomorphic to a disc possibly less a relatively closed polar set, then one may compute the Gaussian curvature directly and find it to be  identically $-4 \pi$.
Conversely,  if the curvature of the Bergman kernel $K(z)|dw|^2$ 
is identically $-4\pi$, then by \eqref{curvK}, $2 \pi K \equiv g $,
which further implies that the Gaussian curvature of the Bergman metric $g (z)|dw|^2$ is identically $-2$.  Thus, by Theorem \ref{3rd}, Part 2, $X$ is biholomorphic to a disc possibly less a relatively closed polar set.

\end{proof}

Under the completeness assumption, we give a proof of Theorem \ref{lower}.
The proof in fact works for any complete metric whose 
Gaussian curvature has a  lower bound $-2$.

\begin{proof} [\bf  Proof of Theorem \ref{lower}] Since the complete Bergman metric on $X$ has Ricci curvature lower bound $-1$, and the Suita metric $c_{\beta}^2(z)|dw|^2$ on $X$ has Gaussian curvature upper bound $-4$, Yau's Schwarz lemma \cite{Yau78} applied to the identity map on $X$ yields that $g  \geq 2 c^2_\beta $.

\end{proof}

The  `only if' part of  Corollary \ref{cor} follows from Theorems \ref{3rd} and \ref{lower}, and 
 the possible relatively closed polar set is empty due to the Bergman completeness.   Alternatively,  the `only if' part could follow from a result of Minda \cite{Min} who proved that   the Poincar\'e metric dominates the Suita metric $c_{\beta}^2(z)|dw|^2$. The converse direction is trivial.

 \medskip

The rest of this section is devoted to the proofs of our results for planar domains in $\mathbb C$.

\begin{proof} [\bf Proof of Theorem \ref{4th}]

By Kellogg's theorem, cf. \cite[Theorem 4.2.5]{R95}, $I= (D \cup I) \cap (\mathbb C \setminus D)$ is a relatively closed polar subset of $D \cup I$, so  the $L^2$ holomorphic functions on $(D \cup I)  \setminus I =D$ will extend holomorphically across $I$. 
Therefore, as domains, $ D \cup I$ and $D$ have the same Bergman kernel and metric. 
 
 Let $J \subset  \partial D$ denote the set of regular boundary points. 
 Then, $\partial D = J \cup I$, where $I$ could possibly be  empty. It is well known that in the one dimensional case the Dirichlet regularity is equivalent to the hyperconvexity.
 By assumption,  the boundary of the domain $ D \cup I$ is exactly $J$, at which the Bergman metric is complete. Since the Gaussian curvature of the Bergman metric on $ D \cup I$  is complete and identically $-2$, Lu's uniformization theorem \cite{Lu} yields that $ D \cup I$ is biholomorphic to a disc. Now the conclusion follows as  the polarity is preserved under biholomorphism.

    \end{proof}

    The so-called $L^2$-domain of holomorphy  is the domain of existence of some $L^2$ holomorphic function, and its  boundary 
   contains no polar part (see \cite{PZ}), so by Theorem \ref{4th} we get

\begin{cor} \label{domain of h} Let $D \subset \mathbb C $ be 
an $L^2$-domain of holomorphy such that the Gaussian curvature of the Bergman metric on $D$ is identically $-2$, and let $I \subset  \partial D$ denote the set of possible irregular boundary points  with respect to the Dirichlet problem.
  If $ D \cup I$ is an open set, then
   $D$ is biholomorphic to a disk.
\end{cor}

 Corollary \ref{domain of h}  improves Corollary 1.4 in \cite{DW} for the case $n=1$ by replacing the Condition ($\star$) there. 
 Any Bergman-complete domain is an $L^2$-domain of holomorphy;  moreover, the converse is not true (see Example \ref{Zal} below).
 Notice that for general planar domains, the assumption $ D \cup I$ being an open set in 
Theorem \ref{4th} and Corollary \ref{domain of h}
does not always hold.

\begin{example} \label{Zal}
Consider the following Zalcman type domain
$$
\mathcal D:= \mathbb D(0, 1) \setminus \left( \bigcup_{k=1}^{\infty} \overline{\mathbb D (x_k, r_k)} \cup \{0\}\right),
$$
where $x_k > x_{k+1}>0$, $\lim_{k\to 0} x_k =0$, $r_k>0,$ $\overline{\mathbb D (x_k, r_k)} \subset \mathbb D(0, 1) \setminus \{0\}$ and $\overline{\mathbb D (x_k, r_k)} \cap \overline{\mathbb D (x_l, r_l)} =\emptyset $ for $k\neq l$.  
One may choose   {(cf. \normalfont\cite{Ju, PZ})} $x_k$ and $r_k$ such that $0$ is the only Bergman-incomplete boundary point with
 $$
     \limsup_{z\to 0} K(z) = \infty  >   \liminf_{z\to 0} K(z).
 $$
Since $
     \limsup_{z\to q} K(z) = \infty
     $
at any $q \in \partial \mathcal  D$, 
a result of Pflug and Zwonek \cite{PZ} says that $\mathcal  D$ is an $L^2$-domain of holomorphy. Moreover, $0$ is the only irregular boundary point of $\mathcal  D$ but $\mathcal  D \cup \{ 0\}$ is not an open set.

\end{example}

  \begin{proof}  [\bf Proof of Theorem \ref{boundary point}]

It suffices to prove the `only if' part.       If there exists $q   \subset \partial D$ such that 
    $$
     \limsup_{z\to q} K(z) < \infty,
     $$
    then a result of Pflug and Zwonek \cite{PZ} says that there exists a neighbourhood $U$ of $q$ such that $ U \setminus D$ is a polar set. Furthermore, $q$ is an irregular boundary point and $\mathcal C=\{q\}$, where $\mathcal C$ is the component of $\partial D$ which contains $q$, cf. \cite[Theorem 4.2.2, Theorem 4.2.3]{R95}. Therefore, all $L^2$ holomorphic functions extend across any such $q$ to the extended domain $\tilde D$, whose Bergman kernel and metric are the same as those of $D$. Since  no  boundary point satisfies \eqref {=>},  $\tilde D$ is Bergman exhaustive,  namely the Bergman kernel function $K(z)$ blowing up to infinity at  $\partial \tilde D$.  By a result of Chen \cite{C00}, we know that the Bergman metric on $\tilde D$ is complete. Since the Gaussian curvature of the complete Bergman metric is identically $-2$, Lu's uniformization theorem \cite{Lu} implies that $\tilde D$ is biholomorphic to a disc.
 Consequently, $ D$ is biholomorphic to a disc possibly less a relatively closed polar set, as  the polarity is preserved under biholomorphism. The proof is then complete in view of Theorem \ref{main}.

\end{proof} 

 Lastly, by the inequality  \eqref{ineq} and a comparison result of Lu \cite{Lu57}, we derive on planar domains the following equivalence of the analytic capacity and the Bergman metric.
 
\begin{cor} \label{equiv} Let $D \subset \mathbb C $ be 
a bounded domain. If the Gaussian curvature of the Bergman metric $g$ on $D$  is identically $-2$, then the analytic capacity $c_B$ on $D$ satisfies
$$
2 c_B^2 \geq    g \geq  c_B^2.
$$ 
 
\end{cor}
 
 Notice that the above first inequality is sharp; it becomes equality when $D$ is biholomorphic to a disc possibly less a relatively closed polar set.
 
\section{Application to disc quotients} \label{Sec6}
We will focus on disc quotients $\Omega:=\mathbb D/\Gamma$, where $\Gamma$ is a finite subgroup
of ${\rm Aut(\mathbb D)}$. Then $\Omega$ is a Stein space with only one isolated singularity (see \cite{ HL, EXX} and the references therein). 
 Let $\kappa_{\Omega} $ and $\omega_\Omega$ be the Bergman kernel and metric on ${\rm Reg}(\Omega)$, respectively. For a local coordinate $w$ on a neighborhood $V \subset {\rm Reg}(\Omega)$, we define
$$
k_{\Omega}(z)|dw|^2 := \kappa_{\Omega}(z)|_{V },\quad  \frac{\partial^2\log k_{\Omega} (z)}{\partial w\partial\overline{w}} |dw|^2 =\omega_{\Omega}(z)|_{V }.
$$
Let $\{\alpha_1, \alpha_2, ... , \alpha_j, ... \}$ be a complete orthonormal basis of the space of $L^2$ holomorphic 1-forms on $\Omega$, and locally write $\alpha_j=a_jdw $.
 Let $\pi: \mathbb D\rightarrow \mathbb D/\Gamma$ be the standard
branched covering map. Then,  $a_j\circ\pi \det \pi^{\prime}=f_j$, and $|\det{\pi^\prime}|^2=1$ since $\Gamma\subset{\rm Aut}~(\mathbb D)$. Thus,
 $
    |a_j\circ \pi|^2=|f_j|^2, \text{ for all } j.
 $
Therefore, by definition
    \begin{equation} \label{Bk}
\pi^\ast k_{\Omega} = \sum_{j=1}^\infty|a_j\circ \pi|^2.
    \end{equation}

On the other hand, write $\pi^\ast \alpha_j=f_j dz$, where $\{f_j\}$ are holomorphic
functions on $\mathbb D\setminus\{0\}$. By the $L^2$ removable singularity
theorem, $\{f_j\}$ can be holomorphically extended to $\mathbb D$.
Moreover, $f_j$ satisfies $$f_j\circ \gamma (z)\det\gamma=f_j(z), \quad
\forall \gamma\in\Gamma,  \forall z\in\mathbb D.$$
Set $A^2_{\Gamma}(\mathbb D):=\{f\in A^2(\mathbb D): f\circ \gamma \det\gamma=f,
\forall \gamma\in \Gamma\}$, which is a closed
 subspace of $A^2(\mathbb D)$. 
 It can be shown (see \cite{HL}) that  $\{\frac{1}{\sqrt{|\Gamma|}} f_j\}_{j=1}^\infty$ is an orthonormal basis of $A^2_{\Gamma}(\mathbb D)$. 
Define the kernel
  $$
    K_{\Gamma}(z, \overline z):=\frac{1}{|\Gamma|}\sum_{j=1}^\infty|f_j(z)|^2.$$
 Then, by \eqref{Bk}, it holds that
   $$
    K_{\Gamma}(z, \overline z) =\frac{1}{|\Gamma|}\pi^\ast k_{\Omega}.
$$
 Consequently,  $\omega_\Gamma:=i\partial\overline\partial \log K_\Gamma(z, \overline z)$ satisfies
 $$
  \omega_\Gamma = \pi^\ast \omega_\Omega.
 $$

\begin{proof} [\bf Proof of Corollary \ref{quotient}] 
For any nontrivial finite subgroup $\Gamma\subset{\rm
Aut}(\mathbb D)$,  assume $|\Gamma|=r \in \mathbb Z \cap [2, \infty) $. 
By  \cite {HL},
$$
K_{\Gamma}(z, \overline w)=\sum_{\gamma\in \Gamma}\frac{K_{\mathbb D}(\gamma z, \overline w)\det\gamma}{|\Gamma|}
=\sum_{k=1}^\infty \frac{rk|z|^{2(kr-1)}}{\pi}= \frac{r |z|^{2(r-1)}}{\pi(1-|z|^{2r})^2},
$$
    where $K_{\mathbb D} $ is the Bergman kernel function on $\mathbb D$. So,
 $$
\pi^\ast k_{\Omega} =   {|\Gamma|} K_{\Gamma}(z, \overline z) =  \frac{r^2 |z|^{2(r-1)}}{\pi(1-|z|^{2r})^2},
$$
and $\pi^\ast \omega_\Omega =     \omega_\Gamma =  i\partial\overline\partial \log K_\Gamma(z, \overline z) =2r^2\frac{|z|^{2(r-1)}}{(1-|z|^{2r})^2} |dw|^2$.
Therefore, 
 $$
 2\pi  (\pi^\ast \kappa_{\Omega}) - \pi^\ast \omega_{\Omega} = \left ( 2\pi  \frac{r^2}{\pi}\frac{|z|^{2(r-1)}}{(1-|z|^{2r})^2}
 - 2r^2\frac{|z|^{2(r-1)}}{(1-|z|^{2r})^2} \right) |dw|^2=0,$$
 which implies that  $2 \pi \kappa_{\Omega} \equiv \omega_{\Omega} $. By Theorem \ref{2nd}, ${\rm Reg}(\Omega)$ is biholomorphic to $\mathbb D \setminus P$, where $P$ is a relatively closed polar set. Since $\Gamma \neq \{id\}$, the set $P$ cannot be empty.

On the other hand, since $\Omega$ has only one isolated singularity, it holds that (cf. \cite[Proposition 1.40]{Ha} )
$$
 \pi_1(\mathbb D \setminus P  ) =  \pi_1({\rm Reg}(\Omega))  =  \pi_1( ( \mathbb D \setminus \{0\} ) / \Gamma ) = \mathbb Z = \pi_1(  \mathbb D \setminus \{0\}    ).
$$ 
 Therefore, the set $P$ only contains a single point. After a possible M\"obius transform, one concludes that $ {\rm Reg}(\Omega)  $ is biholomorphic to $\mathbb D \setminus \{0\}.$

    \end{proof}

Corollary \ref{quotient} strengthens Proposition 5.6 in \cite{HL}, as the Bergman metric on any Riemann surface that is bibiholomorphic to $\mathbb D \setminus \{0\} $ is K\"{a}hler-Einstein.

\subsection*{Funding}

{\fontsize{11.5}{10}\selectfont

The research of the author is supported by AMS-Simons travel grant. This work was supported by the Ideas Plus grant 0001/ID3/2014/63 of the Polish Ministry of Science and Higher Education.
}

\subsection*{Acknowledgements} 
{\fontsize{11.5}{10}\selectfont

The author sincerely thanks Zbigniew B\l{}ocki for suggesting problems around Suita's conjecture. He is also grateful to Professors Bo Berndtsson, \.Zywomir Dinew, L\'aszl\'o Lempert, Song-Ying Li, Takeo Ohsawa and Bun Wong for the valuable communications.}

\bibliographystyle{alphaspecial}

\begin{thebibliography}{HD} 


{\fontsize{11}{11}\selectfont

\bibitem {Be0} {B. Berndtsson}, \emph{The extension theorem of Ohsawa-Takegoshi and the theorem of Donnelly-Fefferman}, Ann. Inst. Fourier (Grenoble) {\bf 46} (1996), 1083--1094.

\bibitem {B1} B. Berndtsson, \emph{Subharmonicity properties of the Bergman kernel and some other functions associated to pseudoconvex domains}, Ann. Inst. Fourier (Grenoble) {\bf 56} (2006), 1633--1662.

\bibitem {B2} B. Berndtsson, \emph{Curvature of vector bundles associated to holomorphic fibrations}, Ann. of Math. {\bf 169} (2009), 531--560.

\bibitem {BL} B. Berndtsson and L. Lempert, \emph{A proof of the Ohsawa-Takegoshi theorem with sharp estimates}, J. Math. Soc. Japan {\bf 68} (2016), 1461--1472.

\bibitem {Bl07} {Z. B\l{}ocki}, \emph{Some estimates for the Bergman kernel and metric in terms of logarithmic capacity}, Nagoya Math. J. {\bf 185} (2007), 143--150.

\bibitem {Bl13} Z. B\l{}ocki, \emph{Suita conjecture and the Ohsawa-Takegoshi extension theorem}, Invent. Math. {\bf193} (2013), 149--158.

\bibitem {Bl14} {Z. B\l{}ocki}, \emph{A Lower Bound for the Bergman Kernel and the Bourgain-Milman Inequality}, Geometric Aspects of Functional Analysis, 53--63, Lecture Notes in Math., 2116, Springer, Cham, 2014.

\bibitem {Bl} {Z. B\l{}ocki}, \emph{Cauchy-Riemann meet Monge-Amp\`{e}re}, Bull. Math. Sci. {\bf 4} (2014), 433--480.

\bibitem {Bl15} {Z. B\l{}ocki}, \emph{Bergman kernel and pluripotential theory}, Analysis, complex geometry, and mathematical physics: in honor of Duong H. Phong, 1--10, Contemp. Math., 644, Amer. Math. Soc., Providence, RI, 2015.

\bibitem {Bl17} {Z. B\l{}ocki}, \emph{Suita Conjecture from the one-dimensional viewpoint}, Analysis meets geometry, 127--133, Trends Math., Birkh\"{a}user/Springer, Cham, 2017.


\bibitem  {C00} B. Chen, \emph{A remark on the Bergman completeness}, Complex Variables Theory Appl. {\bf 42} (2000), 11--15.


\bibitem {Ch} {B.-Y. Chen}, \emph{A remark on an extension theorem of Ohsawa}, Chin. Ann. Math., Ser. A {\bf 24} (2003), 129--134 (in Chinese); translation in Chinese J. Contemp. Math. {\bf 24} (2003), 97--104.



\bibitem {Dem92} {J.-P. Demailly}, \emph{Regularization of closed positive currents and intersection theory}, J. Algebraic Geom. {\bf 1} (1992), 361--409.

\bibitem {Dem12} J.-P. Demailly, \emph{Analytic Methods in Algebraic Geometry}, Surv. Mod. Math., 1, International Press, Somerville, MA; Higher Education Press, Beijing, 2012.



\bibitem {DK} {J.-P. Demailly and J. Koll\'ar}, \emph{Semi-continuity of complex singularity exponents and K\"ahler-Einstein metrics on Fano orbifolds}, Ann. Sci. \'Ecole Norm. Sup. {\bf 34} (2001), 525--556.

\bibitem {DT} {R. X. Dong and J. Treuer}, \emph{Rigidity theorem by the minimal point of the Bergman kernel}, J. Geom. Anal. {\bf 31} (2021), 4856--4864.

\bibitem {DTZ} {R. X. Dong, J. N. Treuer and Y. Zhang}, \emph{Rigidity theorems by capacities and kernels}, 
  \href{https://arxiv.org/abs/2111.10973}{{\color{blue} {arXiv: 2111.10973}}}.
  
  
  
\bibitem {DW} {R. X. Dong and B. Wong}, \emph{Bergman-Calabi diastasis and K\"ahler metric of constant holomorphic sectional curvature}, Pure Appl. Math. Q. (Special Issue in honor of Joseph J. Kohn), to appear.


\bibitem {EXX} {P. Ebenfelt, M. Xiao and H. Xu}, \emph{On the Classification of Normal Stein Spaces and Finite Ball Quotients With Bergman--Einstein Metrics}, Int. Math. Res. Not. (2021), \href{https://doi.org/10.1093/imrn/rnab120}{{\color{blue} {https://doi.org/10.1093/imrn/rnab120}}}.
  
  
  
\bibitem  {FW} {S. Fu and B. Wong}, \emph{On strictly pseudoconvex domains with K\"{a}hler-Einstein Bergman metrics}, Math. Res. Lett. {\bf 4} (1997) 697--703.

  
\bibitem {GZ1} {Q. Guan and X. Zhou}, \emph{Optimal constant problem in the $L^2$ extension theorem}, {C. R. Math. Acad. Sci. Paris} {\bf 350} (2012), 753--756.
 
\bibitem {GZ2} {Q. Guan and X. Zhou}, \emph{An $L^2$ extension theorem with optimal estimate}, {C. R. Math. Acad. Sci. Paris} {\bf 352} (2014), 137--141.

\bibitem {GZ3} {Q. Guan and X. Zhou}, \emph{A solution of an $L^2$ extension problem with optimal estimate and applications}, Ann. of Math. {\bf 181} (2015), 1139--1208.

\bibitem {GZ4} {Q. Guan and X. Zhou}, \emph{A proof of Demailly's strong openness conjecture}, Ann. of Math. {\bf 182} (2015), 605--616. 

\bibitem {GZ5} {Q. Guan and X. Zhou}, \emph{Effectiveness of Demailly's strong openness conjecture and related problems}, Invent. Math. {\bf 202} (2015), 635--676.
 
\bibitem {GZZ1} {Q. Guan, X. Zhou and L. Zhu}, \emph{On the Ohsawa-Takegoshi $L^2$ extension theorem and the twisted Bochner-Kodaira identity}. C. R. Math. Acad. Sci. Paris  {\bf 349} (2011), 797--800. 


  \bibitem {Ha} {A. Hatcher}, \emph{Algebraic Topology}, Cambridge Univ. Press, Cambridge, 2002.
 
\bibitem {HL} {X. Huang and X. Li}, \emph{Bergman-Einstein metric on a Stein space with a strongly pseudoconvex boundary}, Comm. Anal. Geom., to appear.

\bibitem {HX} {X. Huang and M. Xiao}, \emph{Bergman-Einstein metrics, a generalization of Kerner's theorem and Stein spaces with spherical boundaries}, J. Reine Angew. Math. {\bf 770} (2021), 183--203

\bibitem {HX2} {X. Huang and M. Xiao}, \emph{A uniformization theorem for Stein spaces}, Complex Anal. Synerg. {\bf 6} (2020), \href{https://doi.org/10.1007/s40627-020-00044-x}{{\color{blue} {https://doi.org/10.1007/s40627-020-00044-x}}}.
 

\bibitem {Ju} P. Jucha, \emph{Bergman completeness of Zalcman type domains}, Studia Math. {\bf 163} (2004), 71--83.


\bibitem {Lu57} K. H. Look,  \emph{Schwarz lemma in the theory of functions of
several complex variables}, Acta Math. Sinica {\bf 7} (1957), 370--420 (in Chinese).

 

\bibitem {Lu} {Q.-K. Lu}, \emph{On K\"{a}hler manifolds with constant curvature}, Acta Math. Sinica {\bf 16} (1966), 269--281 (in Chinese); translation in Chinese Math.-Acta {\bf 8} (1966) 283--298.





\bibitem {MY} F. Maitani and H. Yamaguchi, \emph{Variation of Bergman metrics on Riemann surfaces}, Math. Ann. {\bf 330} (2004), 477--489.


\bibitem{Min} C. D. Minda, \emph{The capacity metric on Riemann surfaces}, Ann. Acad. Sci. Fenn. Ser. A. I. Math. {\bf 12} (1987), 25--32. 
 
  
 \bibitem {NS} S. Yu. Nemirovski and R. G. Shafikov, \emph{Conjectures of Cheng and Ramadanov}, Uspekhi Mat. Nauk {\bf 370} (2006), 193--194 (in Russian); translation in Russian Math. Surveys {\bf 61} (2006), 780--782.
 
 
\bibitem {O95} {T. Ohsawa}, \emph{Addendum to ``On the Bergman kernel of hyperconvex domains"}, Nagoya Math. J. {\bf 137} (1995), 145--148.

\bibitem {O98} {T. Ohsawa}, \emph{Bergman kernel and the Suita conjecture on Riemann surfaces} (Japanese), The theory of reproducing kernels and their applications (Japanese) (Kyoto, 1998). S$\bar{u}$rikaisekikenky$\bar{u}$sho K$\bar{o}$ky$\bar{u}$roku No. 1067 (1998), 89--95.

\bibitem {O20} {T. Ohsawa}, \emph{A Survey on the $L^2$ Extension Theorems}, J. Geom. Anal. {\bf 30} (2020), 1366--1395.

\bibitem {O20B} {T. Ohsawa}, \emph{A Role of the $L^2$ Method in the Study of Analytic Families}, Bousfield classes and Ohkawa's theorem, 423--435, Springer Proc. Math. Stat., 309, Springer, Singapore, 2020.

 





  
\bibitem {OT} T. Ohsawa and K. Takegoshi, \emph{On the extension of $L^2$ holomorphic functions}, Math. Z. {\bf 195} (1987), 197--204.

\bibitem {Pa} M. P\v{a}un, \emph{Siu's invariance of plurigenera: a one-tower proof}, J. Differential Geom. {\bf 76} (2007), 485--493.


\bibitem {PZ} P. Pflug and W. Zwonek, \emph{$L^2_h$-domains of holomorphy and the Bergman kernel}, Studia Math. {\bf 151} (2002), 99--108.



\bibitem{R95} T. Ransford, \emph{Potential theory in the complex plane}, London Math. Soc. Stud. Texts, 28. Cambridge Univ. Press, Cambridge, 1995.


\bibitem {Si} {Y.-T. Siu}, \emph{The Fujita conjecture and the extension theorem of Ohsawa-Takegoshi}, {Geometric Complex Analysis} (Hayama, 1995), 577--592, World Sci. Publ., River Edge, NJ, 1996.

\bibitem {Si2} {Y.-T. Siu}, \emph{Invariance of plurigenera}, Invent. Math. {\bf 134} (1998), 661--673.

\bibitem {Si02} {Y.-T. Siu}, \emph{Extension of twisted pluricanonical sections with plurisubharmonic weight and invariance of semipositively twisted plurigenera for manifolds not necessarily of general type}, Complex geometry (G\"{o}ttingen, 2000), 223--277, Springer, Berlin, 2002.

\bibitem {Su} N. Suita, \emph{Capacities and kernels on Riemann surfaces}, Arch. Ration. Mech. Anal. {\bf 46} (1972), 212--217.


\bibitem {Yau78} S.-T. Yau, \emph{A general Schwarz lemma for K\"ahler manifolds}, Amer. J. Math. {\bf 100} (1978), 197--203.

\bibitem {Yau} {S.-T. Yau}, \emph{Problem section}, Seminar on Differential Geometry (S.-T. Yau eds.), 669--706, Ann. of Math. Stud., 102,  Princeton Univ. Press, Princeton, NJ, 1982.

 
 
\bibitem{Z} {X. Zhou}, \emph{A Survey on $L^2$ extension problem}, {Complex Geometry and Dynamics}, 291--309, Abel Symp., 10, Springer, Cham, 2015.

\bibitem {GZZ} {L. Zhu, Q. Guan and X. Zhou}, \emph{On the Ohsawa-Takegoshi $L^2$ extension theorem and the Bochner-Kodaira identity with non-smooth twist factor}, {J. Math. Pures Appl.} {\bf97} (2012), 579--601.


 

}

\end{thebibliography}

\fontsize{11}{9}\selectfont

\vspace{0.5cm}

\noindent dong@uconn.edu, 

\vspace{0.2 cm}

\noindent Department of Mathematics, University of Connecticut, Storrs, CT 06269-1009, USA

 \end{document}